\documentclass{extarticle}
\usepackage{amsfonts,amssymb,amsmath,amsthm}
\usepackage{url}
\usepackage{color}
\usepackage{nameref}
\usepackage[dvipsnames]{xcolor}
\usepackage[colorlinks=true,linkcolor=RoyalBlue,citecolor=PineGreen,urlcolor=blue]{hyperref}
\usepackage[margin=1in,letterpaper]{geometry}
\usepackage{bm}
\usepackage[ruled,vlined,linesnumbered]{algorithm2e}
\usepackage[numbers,sort]{natbib}
        \RequirePackage{enumerate} 
          \usepackage[shortlabels]{enumitem}
\usepackage{tikz}
\usepackage{array}
\usepackage{graphicx}
\usepackage{float}
\usepackage{cleveref}
\usepackage{physics}
\usepackage{verbatim}
\usepackage{epsfig}
\usepackage{amsmath}
\usepackage{amssymb}
\usepackage{amsthm}
\usepackage{indentfirst}
\usepackage{xspace}
\usepackage{multirow}
\usepackage{hyperref}
\usepackage{xcolor}
\usepackage{verbatim}
\usepackage{subfigure}
\usepackage{mathrsfs}  
\usepackage{ytableau}
\usepackage{bbm}
\usepackage{rotating}
\usepackage{tabularx}
\usepackage{fancyhdr}
\crefformat{footnote}{#2\footnotemark[#1]#3}

\usepackage{setspace}

\newtheorem{lemma}{Lemma}[section]

\newtheorem{theorem}{Theorem}[section]

\newtheorem{proposition}{Proposition}[section]
\newtheorem{conjecture}{Conjecture}[section]

\theoremstyle{definition}
\newtheorem{example}{Example}[section]
\newtheorem{remark}{Remark}[section]
\newtheorem{definition}{Definition}[section]

\makeatletter
\newcommand{\myitem}[1]{%
\item[#1]\protected@edef\@currentlabel{#1}%
}
\makeatother

\makeatletter
\newcommand\footnoteref[1]{\protected@xdef\@thefnmark{\ref{#1}}\@footnotemark}
\makeatother

\title{HOMFLY Polynomials of Pretzel Knots}

\author{William Qin\footnote{{wqin2008@gmail.com}, Millburn High School, Millburn, USA}}

\date{}

\newcommand\dsq{}
\def\dsq(#1:#2){
  \draw (#1,#2) -- (#1+0.5,#2) -- (#1+0.5,#2+0.5) -- (#1,#2+0.5) -- cycle;
}

\newcommand\lc{}
\def\lc(#1:#2){
  \draw (#1+0.5,#2+0)--(#1+0.1,#2+0.8);
  \draw (#1+-0.1,#2+1.2)--(#1+-0.5,#2+2);
  \draw (#1+-0.5,#2+0)--(#1+0.5,#2+2);
}

\newcommand\rc{}
\def\rc(#1:#2){
  \draw (#1-0.5,#2+0)--(#1+-0.1,#2+0.8);
  \draw (#1+0.1,#2+1.2)--(#1+0.5,#2+2);
  \draw (#1+0.5,#2+0)--(#1+-0.5,#2+2);
}

\newcommand\nc{}
\def\nc(#1:#2){
  \draw (#1+0.5,#2+0)--(#1+0.5,#2+2);
  \draw (#1-0.5,#2+0)--(#1+-0.5,#2+2);
}

\begin{document}
\maketitle

\begin{abstract}
HOMFLY polynomials are one of the major knot invariants being actively studied. They are difficult to compute in the general case but can be far more easily expressed in certain specific cases. In this paper, we examine two particular knots, as well as one more general infinite class of knots.

From our calculations, we see some apparent patterns in the polynomials for the knots $9_{35}$ and $9_{46}$, and in particular their $F$-factors. These properties are of a form that seems conducive to finding a general formula for them, which would yield a general formula for the HOMFLY polynomials of the two knots.

Motivated by these observations, we demonstrate and conjecture some properties both of the $F$-factors and HOMFLY polynomials of these knots and of the more general class that contains them, namely pretzel knots with 3 odd parameters. We make the first steps toward a matrix-less general formula for the HOMFLY polynomials of these knots.
\end{abstract}

\section{Introduction}

HOMFLY polynomials are subcategory of knot polynomial, that generalizes the Jones and Alexander polynomials. In the fundamental representation, they satisfy the skein relation $\frac{1}{A} {\cal{H}}(L_{+})-A{\cal{H}}(L_{-})=(q-q^{-1}){\cal{H}}(L_0)$.

HOMFLY polynomials can be generalized to non-fundamental representations of $SU(2)$, creating the ``colored'' HOMFLY polynomials. While the normal HOMFLY polynomials are relatively easy to compute and have been computed for hundreds of knots in repositories like \cite{katlas}, colored HOMFLY polynomials are difficult to compute in general for any non-(anti)symmetric representations, and computations even in the simplest such representation, namely that corresponding to the smallest L-shaped Young diagram, are difficult to do for more than a few knots.

The crossing changing formula (skein relation) of the HOMFLY polynomial in the fundamental representation, or the original HOMFLY polynomial, was initially constructed to concern variables $t$ and $\nu$\cite{lin2006hecke}. However, when expanded to other representations, the HOMFLY polynomials are quantum invariants associated with irreducible representations of the quantum group $U_{q}({sl}_{N})$, and it is then possible to express it in terms of variables $q$ and $A=q^{N}$\cite{lin2006hecke}. This is the formulation we use in this text.
    
Certain methods have been developed in the past that can compute general colored HOMFLY polynomials. However, each of these general methods relies on the use of matrix multiplication. The most general such method, using quantum R-matrices and braid diagrams is not feasible, despite working for all knots and all representations. This is because it involves the multiplication of square matrices with dimensions increasing as the product of the size of the representation and of the number of strands. The matrix entries are sufficiently complex that this multiplication takes on the order of minutes even for very small representations and knots, and increases very quickly from there. Even less general methods which rely on matrices, despite having smaller matrices, become infeasible at relatively low representation size, and often work only for (anti)symmetric representations or rectangular representations. Therefore, it is desirable to find some method of computing HOMFLY polynomials that does not involve matrix multiplication so that it is feasible to compare knots using these polynomials, even if only for specific knots or small classes of knots.
    
Certain classes of knots such as torus and twist knots already have short general formulas for (anti)symmetric representations and even for rectangular representations but a general, simple formula for HOMFLY polynomials in all representations is currently out of reach for most if not all knots. However, there are easier methods for computing the HOMFLY polynomial of a certain class of knot called an aborescent knot, which allow us to compute many colored knot polynomials for these knots. Therefore, in this paper we examine the two smallest aborescent knots that are neither twist nor torus, and attempt to find a formula for their HOMFLY polynomials. Both of these knots, $9_{35}$ and $9_{46}$ in the Rolfsen table \cite{katlas}, are part of a more general class of knot called a pretzel knot, which we define later in this paper. In particular, they are both pretzel knots with 3 odd parameters, which are $(3,3,3)$ and $(3,3,-3)$ respectively.

One of the new methods, the techniques of which we use heavily in this paper, is the Racah matrix approach, which we use in this text to derive the non-fundamental knot polynomials of these two knots, and we use the general formula for pretzel knots, derived by this same approach, in the latter half of the paper.

We do not yet know a general formula for the HOMFLY polynomials of these knots that does not use matrices, which as mentioned before are slow to the point of computational infeasibility, but we begin to make progress towards such an explicit formula at least for symmetric and anti-symmetric representations, particularly using the differential expansion of \cite{defect}.

\section{Definitions}

\subsection{Representations and symmetric functions}

We use the standard notation $$\{x\}=x-\frac{1}{x},$$ $$D_j=\frac{\{A q^j\}}{\{q\}},$$ and $$[n]=\frac{\{q^n\}}{\{q\}} = q^{n-1}+q^{n-3}+...+q^{1-n},$$ 
 which is equal to the character of the $n$-dimensional representation of $SL(2)$.

We also use the intuitive notation $$[n]! = \prod_{i=1}^{n}[i].$$

Representations of $GL(N)$ are parametrized by Young diagrams with at most $n$ rows. Their characters are symmetric functions in $x_1,...,x_N$, which are called Schur functions $s_\lambda(x_1,...,x_N)$.  We denote these Young diagrams as [$c_1$,$c_1$,$\cdots$,$c_n$], where the $c_i$ are the number of boxes in the $i^{\text{th}}$ column.

Schur functions can be computed by the following determinantal identity:
\[
s_\lambda=\det_{i,j=1}^{n}(h_{\lambda_i-i+j}),
\]
where $h_n = s_{[n]}$ can be computed using the identity
\[
\sum_n h_n t^n = e^{\sum_n \frac{p_n}{n} t^n}.
\]
It is often useful, in practice, by way of the Taylor series of $e^x$, to express them more explicitly as 

\[
h_n=\sum_{x_1+2x_2+3x_3+\cdots = n}\frac{p_1^{x_1}}{1^{x_1}(x_1!)}\frac{p_2^{x_2}}{2^{x_2}(x_2!)}\frac{p_3^{x_3}}{3^{x_3}(x_3!)}\cdots.
\]
Here $p_n$ are the power-sum polynomials
\[
p_k = \sum_i x_i^k.
\]
In order to consider stabilization in the limit $n\to\infty$, it is convenient to introduce $A$ such that
$$
A = q^N.
$$
HOMFLY polynomials are Laurent polynomials in $A$ and $q$, and the dependence on $N$ goes through this relation.

At the special point $p_i^* = \frac{\{A^i\}}{\{q^i\}},$ Schur functions have the hook-length factorization
\[
s_\lambda(p_i^*) = \prod_{(i,j) \in \lambda} \frac{\{Aq^{i-j}\}}{\{q^{\text{hook}_{ij}}\}},
\]

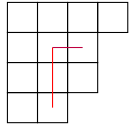
\begin{figure}[H]
    \centering
    \begin{tikzpicture}[scale=0.8]
    \dsq(1:1);
    \dsq(1.5:1);
    \dsq(2:1);
    \dsq(2.5:1);
    \dsq(1:0.5);
    \dsq(1.5:0.5);
    \dsq(2:0.5);
    \dsq(1:0);
    \dsq(1.5:0);
    \dsq(2:0);
    \dsq(1:-0.5);
    \dsq(1.5:-0.5);
    \draw[red] (1.75, -0.25) -- (1.75, 0.75);
    \draw[purple] (2.25, 0.75) -- (1.75, 0.75);
    \end{tikzpicture}
    \caption{A hook of length 4 in the Young diagram [4,4,3,1], beginning at at box (1,1)}
    \label{fig:hook}
\end{figure}

\noindent and this expression coincides with the $\lambda$-colored HOMFLY polynomial of the unknot. Throughout the paper we divide all $\lambda$-colored HOMFLY polynomials by this quantity, so the unknot HOMFLY polynomials are always equal to 1.

\subsection{Pretzel Knots}

In this text, we primarily consider pretzel knots (discussed in Section~\ref{pretzel}) with $3$ odd parameters $a,b,c$.

\begin{definition}
We denote the HOMFLY polynomial of a 3-parameter pretzel knot in some symmetric representation $[r]$ as $\mathcal{H}_r(a,b,c) = {\cal{H}}(a,b,c,r) = \chi_{[r,0]}\sum_{x=0}^{r}\frac{1}{S_{0,x}}(\overline{S}\cdot\overline{T}^{a}\cdot S)_{0,x}(\overline{S}\cdot\overline{T}^{b}\cdot S)_{0,x}(\overline{S}\cdot\overline{T}^{c}\cdot S)_{0,x}$.
\end{definition}

This formula was derived using Racah matrices in equation 22 of \cite{col} in terms of different Racah matrices, though we change conventions here and therefore modify the formula slightly to fit these new conventions.

We denote ${\cal{H}}_r(a,b,c)$ as $Q(c,r)$ for some arbitrary fixed $a,b$. In this text we typically set $a=b=1$, but when we use the notation $Q(c,r)$ our only restriction on $a,b$ is that within a formula, they are always the same constants. We may also write $Q(a,b,c,r)$, when we wish to specify $a,b$.

\begin{definition}
We recursively define the $n^{\text{th}}$ difference of a genus-2 pretzel knot's HOMFLY polynomial $Q^n(c, r) = Q^{n-1}(c+2, r) - Q^{n-1}(c, r)$ for positive integers $n$, and $Q^{0}(c,r)=Q(c,r)$. We may also write $Q^n(a,b,c,r)$, when we wish to specify $a,b$.
\end{definition}

In particular, $c$ is always odd, and the ``next'' $n^{th}$ difference, where we talk about it, is actually $Q^{n}(c+2,r)$.

We denote by $X(P(x)), \text{ for all } P(x)\in \mathbb{Z}[x]$, the factor of $P(x)$ that cannot be factored further with the highest degree. In this paper we use this only in the context of $X(Q^{n}(c,r))$.

\begin{remark}
We can alternatively define the $n^{\text{th}}$ difference as $X(Q^{n-1}(c+2, r)) - X(Q^{n-1}(c, r))$. We denote this expression $Q^{(n)}(c,r)$ in this text. When we list the computed differences, we use this definition (occasionally omitting the parentheses), however in our proofs we use the former for simplicity. Thus far it seems that the two definitions are related in that $Q^{n}(c,r)=Q^{(n)} (c,r)\prod_{i=1}^{n-1}\frac{Q^{i}(c,r)}{X(Q^{i})c,r_)}.$
\end{remark}

\subsection{HOMFLY polynomials}

When we refer to $F$-factors, we are referring to the factors in the differential expansion of the HOMFLY polynomial as used in \cite{twist}. While it is used there for twist knots, the formula
\begin{equation}
H_{[r]} = 1 + \sum_{s=1}^r  \frac{[r]!}{[s]![r-s]!}\,F_s(A|q)\,
\prod_{j=0}^{s-1} \{Aq^{r+j}\}\{Aq^{j-1}\}
\label{twir}
\end{equation}
applies more generally for defect-zero knots, which was used to derive them for the knots considered here. Both of the knots considered in Section~\ref{conj} are defect-zero and so we can find and analyze their $F$-factors given their HOMFLY polynomials, using the methods outlined in \cite{pre}. 

Also, we get from \cite{defect} that the defect $\delta^{\cal K}$ of the differential expansion depends on the degree of the Alexander polynomial in a way that allows us to compute the $F$-factors examined in Section~\ref{conj}. If we denote the Alexander polynomial by $Al^{\cal K}, then$

$$
\delta^{\cal K} = \frac{1}{2} \text{Power}_{q^2}(Al^{\cal K}) - 1.
$$

This is very convenient because many Alexander polynomials are readily available from \cite{katlas} (where we substitute $t=q^2$), and are also easily computed as $H_{[1]}^{\cal K}(A=1,q)$ from only the HOMFLY polynomial in the fundamental representation.

In particular, from this property, we can conclude that both $9_{35}$ and $9_{46}$ have defect $0$, which allows us to use Formula (~\ref{twir}).

\begin{example}
$Al^{9_{35}}=7q^2-13+\frac{7}{q^2}$, hence $\text{Power}_{q^2}(Al^{9_{35}})=2$ and $\delta^{\cal K}=0$.
\end{example}

\begin{example}
$Al^{9_{46}}=-2q^2+5-\frac{2}{q^2}$, hence $\text{Power}_{q^2}(Al^{9_{46}})=2$ and $\delta^{\cal K}=0$.
\end{example}

\begin{example}
$Al^{9_{1}}=q^8-q^6+q^4-q^2+1-\frac{1}{q^2}+\frac{1}{q^4}-\frac{1}{q^6}+\frac{1}{q^8}$, hence $\text{Power}_{q^2}(Al^{\cal K})=8$ and $\delta^{\cal K}=3$.
\end{example}

All of the pretzel knots we consider here can be trivially confirmed with the tables in \cite{defect} to be defect-zero, however whether this is true in general for genus $2$ pretzel knots is not yet known.

\section{Pretzel Knots}\label{pretzel}

Pretzel knots of genus $g$ are knots created by connecting pairs of crossing strands as in Figure~\ref{fig:pre}. We deal with genus 2 pretzel knots in this text (or equivalently, pretzel knots with 3 parameters). Each parameter (the $n_{i}$) determines the number of crossings in each pair.

\begin{figure}[H]
    \centering
    \includegraphics[scale=0.2]{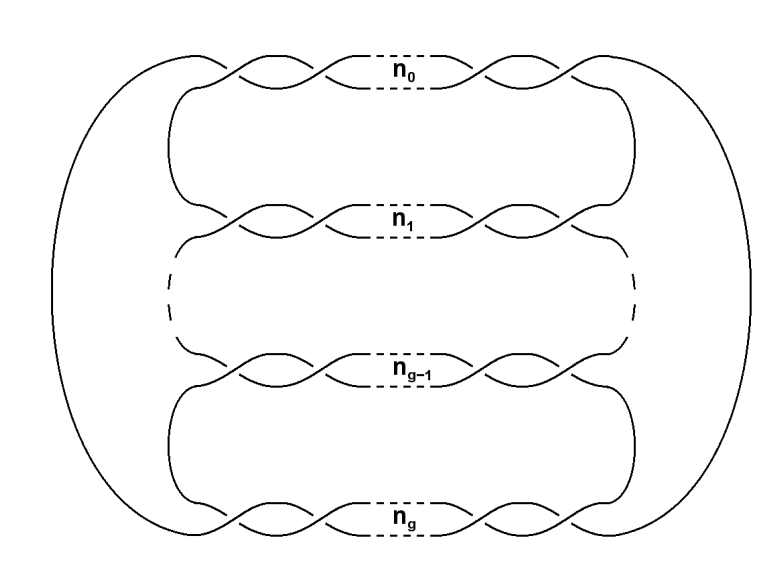}
    \caption{An illustration of a genus $g$ pretzel knot}
    \label{fig:pre}
\end{figure}

From \cite{pretzel}, we have the following two formulas about $S$ and $\bar{S}$ Racah matrices, which are used in the general formula for computing HOMFLY polynomials of genus $g$ knots.

\begin{equation}\label{eq:S}S_{km} = \sum_{j=\max(r+m,r+k)}^{\min(r+k+m,2r)}\sigma_{km}(j)\cdot \sqrt{\frac{[2m+1] \Delta_k}{[2k+1]\chi_{[r+m,r-m]}}}\cdot\frac{G(r-m)G(j+1)}{G(r+k+1)G(j-r-m)},\end{equation}
and
\begin{equation}\label{eq:sbar}\bar S_{km} = \frac{[r+1]!}{\prod_{i=0}^{r-1} D_i}\cdot\sum_{j=\max(r+m,r+k)}^{\min(r+k+m,2r)}\sigma_{km}(j)\cdot \sqrt{\frac{\Delta_k\Delta_m }{[2k+1][2m+1]}}\cdot\frac{G(r+1)G(j+1)}{G(r+k+1)G(r+m+1)G(r+k+m-j)}.\end{equation}

In these formulas,

\begin{multline}\nonumber
\sigma_{km}(j) = (-1)^{r+k+m} \sqrt{[2k+1][2m+1]} \cdot \frac{([k]![m]!)^2[r-k]![r-m]!}{[r+k+1]![r+m+1]!} \cdot\\\frac{(-1)^j[j+1]!}{[2r-j]!([j-r-k]![j-r-m]![r+k+m-j]!)^2}, \text{ for all } 0\leq k,m \leq r,
\end{multline}

$$
G(n)=\prod_{j=1}^n \frac{\{A q^{j-2}\}}{\{q^j\}},
$$

$$
\Delta_m=\frac{D_{2m-1}}{D_{-1}}G(m)^2
$$

and

$$
\chi_{[r+m,r-m]}=\frac{G(r+m+1)G(r-m)}{D_{-1}}[2m+1].
$$

We also have that \begin{equation}\label{eq:T}\bar{T}_{km}=\begin{cases}0&k\neq m\\(-q^{m-1}A)^{m}&k=m.\end{cases}\end{equation}

Before we begin, we prove a small lemma about genus-2 pretzel knots.

\begin{lemma}\label{lemma:permute}
For any genus-2 pretzel knot with parameters $a,b,c$, the knot is invariant under any permutation of the parameters.
\end{lemma}
\begin{proof}
We can prove that the genus-2 pretzel knots are invariant under rotation of the parameters as well as order reversal, which will allow us to demonstrate that the knot is invariant under any permutation of the parameters.

For rotation of parameters, we use the Figure ~\ref{fig:rot}.

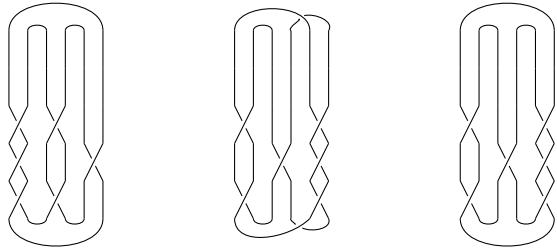
\begin{figure}[H]
    \centering
    \begin{tikzpicture}[scale=0.25]
    \lc(-1:0);
    \lc(-1:2);
    \lc(-1:4);
    \lc(1:0);
    \lc(1:4);
    \lc(3:2);
    \nc(1:2);
    \nc(3:0);
    \nc(3:4);
    \nc(3:6);
    \nc(1:6);
    \nc(-1:6);
    \nc(3:8);
    \nc(1:8);
    \nc(-1:8);
    \draw[bend right=90] (2.5,10) to (1.5,10);
    \draw[bend right=90] (0.5,10) to (-0.5,10);
    \draw[bend right=90] (3.5,10) to (-1.5,10);
    \draw[bend right=270] (2.5,0) to (1.5,0);
    \draw[bend right=270] (0.5,0) to (-0.5,0);
    \draw[bend right=270] (3.5,0) to (-1.5,0);
    
    \lc(15:0);
    \lc(15:2);
    \lc(15:4);
    \lc(11:0);
    \lc(11:4);
    \lc(13:2);
    \nc(11:2);
    \nc(13:0);
    \nc(13:4);
    \nc(13:6);
    \nc(11:6);
    \nc(15:6);
    \nc(13:8);
    \nc(11:8);
    \nc(15:8);
    \draw[bend right=90] (12.5,10) to (11.5,10);
    \draw[bend right=90] (13.9,10.6) to [out=135] (13.5,10);
    \draw[bend left=315] (15.5,10) to [out=270] (14.2,10.8);
    \draw[bend right=90] (14.5,10) to (10.5,10);
    \draw[bend right=270] (12.5,0) to (11.5,0);
    \draw[bend right=270] (14.5,0) to [out=45](10.5,0);
    \draw[bend right=270] (13.7,-0.3) to [out=225] (13.5,0);
    \draw[bend left=45] (15.5,0) to [out=90] (14.1,-0.5);
    
    \lc(23:0);
    \nc(23:2);
    \lc(23:4);
    \nc(25:0);
    \nc(25:4);
    \lc(27:2);
    \lc(25:2);
    \lc(27:0);
    \lc(27:4);
    \nc(27:6);
    \nc(25:6);
    \nc(23:6);
    \nc(27:8);
    \nc(25:8);
    \nc(23:8);
    \draw[bend right=90] (26.5,10) to (25.5,10);
    \draw[bend right=90] (24.5,10) to (23.5,10);
    \draw[bend right=90] (27.5,10) to (22.5,10);
    \draw[bend right=270] (26.5,0) to (25.5,0);
    \draw[bend right=270] (24.5,0) to (23.5,0);
    \draw[bend right=270] (27.5,0) to (22.5,0);
    \end{tikzpicture}
    \caption{Parameter rotation with the knot described by (3,2,1)}
    \label{fig:rot}
\end{figure}

Flipping an intersection does not change the intersection type, which allows us to perform the transformation shown, keeping all of the intersections of the same type, by pulling the first two strands over the rest of the knot, and then flipping the formerly first pair of strands.

To prove that flipping the order of the parameters does not change the not, we need only to flip the knot, which does not change any of the intersection types but reverses the order of the parameters.

For the sake of completeness we list the sequences of these two operations that generate each parameter permutation, with F representing a flip and R representing a rotation.

\begin{tabular}{|c|c|}\hline
    Permutation & Sequence \\\hline
    (a,b,c) & Identity \\\hline
    (a,c,b) & RF \\\hline
    (b,a,c) & RRF \\\hline
    (b,c,a) & R \\\hline
    (c,a,b) & RR \\\hline
    (c,b,a) & F \\\hline
\end{tabular}

Therefore, we can permute the parameters of a genus-2 pretzel knot and the knot will remain the same.
\end{proof}
\begin{remark}
This means also that the HOMFLY polynomial is invariant under these changes. Additionally, the components of this proof also hold for arbitrary genus but except in genus 2 do not prove the lemma.
\end{remark}
    
\section{$9_{35}$ and $9_{46}$}\label{conj}

In this section we consider the smallest defect-zero knots that are neither twist knots nor torus knots, which are $9_{35}$ and $9_{46}$. Both of these knots are pretzel knots with 3 parameters, for which an explicit F-factor formula is not yet known. $F$-factors here are parts of the differential expansion formula for HOMFLY polynomials, as described in \cite{defect}. Both of these knots are small enough to explicitly calculate some of the HOMFLY polynomials, and we are therefore able to compute some $F$-factors to make some conjectures as to their properties. We present limited examples of the conjectures in the text, and provide more $F$-factors in the ancillary files and in the GitHub repository for both knots.

\subsection{$9_{35}$}

We can use the formula for HOMFLY polynomials of pretzel knots in \cite{pre} to get the following $F$-factors. We find the $F$-factors by using Formula~\ref{twir} after computing the HOMFLY polynomials. $9_{46}$ is the pretzel knot $(3,3,3)$. We write $F_i(A\mid q)$ as $F_i$ throughout this text.

We get

$$
F_1=-A^2 (A^6+3 A^4+2 A^2+1),
$$

\begin{multline}\nonumber
F_2=A^4 q^2 (A^{12} q^{12}+3 A^{10} q^{10}+3 A^{10} q^8+5 A^8 q^8+5 A^8 q^6+ \\ 
+3 A^8 q^4+4 A^6 q^6+6 A^6 q^4+3 A^6 q^2+A^6+3 A^4 q^4+4 A^4 q^2+3 A^4+2 A^2 q^2+2 A^2+1),
\end{multline}

\begin{multline}\nonumber
F_3=-A^6 q^6 (A^{18} q^{36}+3 A^{16} q^{32}+3 A^{16} q^{30}+3 A^{16} q^{28}+5 A^{14} q^{28}+\\
8 A^{14} q^{26}+11 A^{14} q^{24}+6 A^{14} q^{22}+3 A^{14} q^{20}+7 A^{12} q^{24}+9 A^{12} q^{22}+18 A^{12} q^{20}+\\
16 A^{12} q^{18}+12 A^{12} q^{16}+3 A^{12} q^{14}+A^{12} q^{12}+6 A^{10} q^{20}+10 A^{10} q^{18}+19 A^{10} q^{16}+\\
19 A^{10} q^{14}+18 A^{10} q^{12}+9 A^{10} q^{10}+3 A^{10} q^8+5 A^8 q^{16}+8 A^8 q^{14}+17 A^8
   q^{12}+
   17 A^8 q^{10}+\\ 17 A^8 q^8+8 A^8 q^6+3 A^8 q^4+4 A^6 q^{12}+6 A^6 q^{10}+12 A^6 q^8+12 A^6 q^6+9 A^6 q^4+3 A^6 q^2+\\  A^6+
   3 A^4 q^8+4 A^4 q^6+7 A^4 q^4+4 A^4 q^2+3 A^4+2 A^2 q^4+2 A^2 q^2+2 A^2+1).
\end{multline}


For all positive integers $i$

$$
F_i' = A^2 q^{2i} F_i
$$

When fully factored, we conjecture that the $F_i'$ are such that the factor of $A$ and $q$ is the same for $F_{i-1}'$ and $F_i$, for all positive integer $i>2$. It can be verified for all integers $2\leq i\leq 7$ that $(1 + A^2 q^{2(i-1)}) \mid F_i + F'_{i-1}$. (See ancillary files) This leads us to the following hypothesis.

\begin{conjecture}
For the knot $9_{35}$, and positive integers $i>1$, $(1 + A^2 q^{2(i-1)}) \mid F_i + F'_{i-1}$.
\end{conjecture}

\subsection{$9_{46}$}

$9_{46}$ is the pretzel knot $(3,3,-3)$. We get

$$
F_1=A^2 (A^2+1),
$$

$$
F_2=A^4 (A^4 q^8+A^2 q^6+A^2 q^4+1),
$$

$$
F_3=A^6 q^4 (A^2 q^4+1) (A^4 q^{16}+A^2 q^{10}+A^2 q^8-q^6+q^2+1).
$$

We see that both $F_1$ and $F_3$ have a factor of $1 + A^2 q^{2(i-1)}$. This also holds true for $F_5$ and $F_7$, giving us the following conjecture.

\begin{conjecture}
For the knot $9_{46}$, and odd positive integer $i>1$, $(1 + A^2 q^{2(i-1)}) \mid F_i$.
\end{conjecture}

After applying the transformation $A\implies A q^2$ to $F_i$, multiplying by $\frac{A^2}{q^4}$, and denoting the result $F^*_i$, we find that $q^2 (q^4 - 1) \mid \frac{F_3}{1 + A^2 q^{4}}-F^*_2,$ because $F_3-F^*_2=A^6 q^6- A^6 q^{10}$. In particular, we can also verify that for odd integers $3 \leq i \leq 7$, $q^2(q^{2(i-1)}-1) \mid \frac{F_i}{1 + A^2 q^{2(i-1)}}-F^*_{i-1}$. This leads to the following conjecture.

\begin{conjecture}
For the knot $9_{46}$, and odd $i>1$, $q^2(q^{2(i-1)}-1) \mid \frac{F_i}{1 + A^2 q^{2(i-1)}}-F^*_{i-1}$.
\end{conjecture}

If the above two conjectures are true, this could help in deriving a formula for arbitrary $F$-factors. For all 3 conjectures in this section, the result after dividing by the stated factor seems to begin similar for all $i$ though they differ after some terms, and we suspect that for high $i$, the remaining parts of each quantity may converge to some infinite polynomial.

\section{Factorization Of General $n^{\text{th}}$ Differences}

The $n^{\text{th}}$ differences for general pretzel knots have many properties making them useful for computing HOMFLY polynomials. In particular, these differences tend to factor nicely as products of quantum numbers, and they are constructed in a similar way to the differential expansion of \cite{defect}. As at least low-valued 3-parameter pretzel knots seem to all be defect-zero, as mentioned in \ref{pretzel}, we can likely use the $n^{\text{th}}$ differences and the differential expansion formula to generate a relatively simple formula for HOMFLY polynomials in at least symmetric and anti-symmetric representations.

In this section, where applicable we fix $a,b=(1,1)$, however the first and second subsections outline some properties that hold regardless of $a,b$. In these subsections $a,b$ are arbitrary fixed constants.

Additionally, in this section we assume that $a,b,c$ are all odd, as this guarantees a knot rather than a link.



\subsection{Conjectures}

Here we give some conjectures that are not proven here and while not integral to the main result, could be helpful in the future if proven.

\begin{conjecture}\label{lemma:mono}
$\frac{Q^{1}(c,r)}{X(Q^{1}(c,r))}=\frac{Q^{1}(c+1,r)}{X(Q^{1}(c+1,r))}$.
\end{conjecture}

\begin{remark}
Conjecture~\ref{lemma:mono}, if true, can easily be extended to show that every $\frac{Q^{1}(c,r)}{X(Q^{1}(c,r))}$ is a constant with respect to $c$, by repeated application.
\end{remark}

\begin{remark}
There are simple counterexamples for higher differences; in a later section we address what occurs when we do not take the largest factor. All that results is that previous factors carry over, but it does turn out to be useful also to study what these factors are, despite corresponding less directly to $F$-factors.
\end{remark}

\subsection{General Properties}

We begin by proving the following proposition. This provides a motivation for a general result about the first differences.

\begin{proposition}
For the representation $[1]$, $(A-q)(A+q)(Aq-1)(Aq+1)\mid Q^{1}(m,1),\text{ for all } a,b,m.$
\end{proposition}
\begin{proof}
   The proposition is equivalent to the statements that $$\begin{array}{c}Q(2m+1,r)-Q(2m-1,r)\mid_{A=q}=0, \\ Q(2m+1,r))-Q(2m-1,r))\mid_{A=-q}=0, \\ Q(2m+1,r))-Q(2m-1,r))\mid_{A=\frac{1}{q}}=0, \\ Q(2m+1,r))-Q(2m-1,r))\mid_{A=-\frac{1}{q}}=0\end{array}$$ based on our definition of $Q^{1}(m,1)$ and the fact that any polynomial is zero exactly when one if its polynomial factors are. We know that $${\cal{H}}(a,b,c,r) = Q(c,r)=\chi_{[1,0]}\sum_{x=0}^{r}\frac{1}{S_{0,x}}(\overline{S}\cdot\overline{T}^{a}\cdot S)_{0,x}(\overline{S}\cdot\overline{T}^{b}\cdot S)_{0,x}(\overline{S}\cdot\overline{T}^{c}\cdot S)_{0,x}.$$
   
   Evaluating Equation~(4) from \cite{pretzel}, we get that with representation $[1],$ $$S=\begin{pmatrix}\sqrt{\frac{(A-q)(A+q)}{(A^2-1)(q^2+1)}} & \sqrt{\frac{(Aq-1)(Aq+1)}{(A^2-1)(q^2+1)}}\\ \sqrt{\frac{(Aq-1)(Aq+1)}{(A^2-1)(q^2+1)}} & \sqrt{\frac{(A-q)(A+q)}{(A^2-1)(q^2+1)}}\end{pmatrix}.$$ 
   At $A=\pm q$, $S_{0,1}=1$ and $S_{0,0}=0$, and similarly for $A=\pm\frac{1}{q}$, $S_{0,1}=0$ and $S_{0,0}=1$. We consider first the case where $A=\pm q$.
   
   If $A=\pm q$, then as $S_{0,0}=S_{1,1}$ and $S_{0,1}=S_{1,0}$, this means that $S=I$.
   
   If $A=\pm\frac{1}{q}$, $S$ becomes the anti-diagonal identity matrix for the same reasoning.
   
   By evaluating Equation (5) of \cite{pretzel}, we get that $$\overline{S}=\begin{pmatrix}\frac{A(q^2-1)}{(A^2-1)q} & \frac{A(q^2-1)\sqrt{\frac{(A-q)(A+q)(Aq-1)(Aq+1)}{A^2(q^2-1)^2}}}{(A^2-1)q}\\ \frac{A(q^2-1)\sqrt{\frac{(A-q)(A+q)(Aq-1)(Aq+1)}{A^2(q^2-1)^2}}}{(A^2-1)q} & \frac{A(q^2-1)}{(A^2-1)q}\end{pmatrix}.$$
   
   Clearly $\overline{S}_{0,1}=\overline{S}_{1,0}=0$ at each of the $4$ points. At $A=\pm q$, it is trivial to verify that the diagonal entries are $\pm 1$, and for $A=\pm\frac{1}{q}$, they are $\mp 1$.
   
   Because $\overline{T}=\begin{pmatrix}1 & 0\\0 & -A\end{pmatrix}=\begin{pmatrix}1 & 0\\0 & \mp\frac{1}{q}\end{pmatrix}$, for $A=\pm\frac{1}{q}$, we find that each of the terms of the form $\overline{S}\cdot\overline{T}^{m}\cdot S$ becomes exactly $\begin{pmatrix}0 & (\mp 1)^m\\\frac{1}{q^m} & 0\end{pmatrix}$. As $m$ is odd, this is exactly $\begin{pmatrix}0 & \mp 1\\\frac{1}{q^m} & 0\end{pmatrix}$. Because we are only looking at the elements $(1,x)$ of this matrix, we are left with only a contribution of $\mp 1$ and $0$ for $x=2,1$ respectively. Combined with our previously computed values for $S_{1,0}$ and $S_{0,0}$ gives a final value for the entire polynomial of $\mp 1$ completely regardless of $a,b,c$.
   
   For $A=\pm q$, by going through the same process with our already computed values, we get that the final polynomial at this point is exactly $\pm 1$ regardless of $a,b,c$.
   
   Regardless of what the constant value is, in all $4$ cases the $Q(c,r)$ are constant for all odd integers $c$, and so it is also equal to this value for both $c=2m+1$ and $c=2m-1$. Therefore, upon subtracting the two polynomials, we get $0$ as the $1^{\text{th}}$ difference. Recalling that we chose the four points specifically because $(A+q)(A-q)(Aq-1)(Aq+1)\mid Q^{1}(m,1)\text{ if and only if } Q^{1}(m,1)$ is zero at all of those points, we can conclude that in fact $(A+q)(A-q)(Aq-1)(Aq+1)\mid Q^{1}(m,1)$, as desired.
   \end{proof}

\begin{remark}
Both of $(A-q)(A+q)$ and $(Aq+1)(Aq-1)$ are quantum numbers up to a monomial factor. In particular, they are $\{\frac{A}{q}\}$ and $\{Aq\}$, respectively. In particular, each of the $\frac{Q(c,r)}{X(Q(c,r))}$ seems to be a product of quantum numbers in a simple way outlined in Conjecture~\ref{lemma:mono}, though proving this is for a future work.
\end{remark}

This proposition can be easily extended to all representations $[r]$, as outlined in Theorem~\ref{thm:main}. We first introduce two lemmas.

\begin{lemma}\label{lemma:s1}
The first row of the Racah matrix $S_{0,m} = \begin{cases}1&m=r\\0&m\neq r\end{cases}$ at $A=\pm q$ and $A=\pm \frac{1}{q}$, and the first row of $\bar{S}_{0,m} = \begin{cases}1&m=0\\0&m\neq 0\end{cases}$ at $A=\pm q$ and $A=\pm \frac{1}{q}$.
\end{lemma}
\begin{proof}
To prove this, we use Equation~\ref{eq:S}.

It is clear that $$A=\pm q\implies G(i)=\prod_{j=1}^n\frac{\{\pm q^{j-1}\}}{\{q^{j}\}}=\frac{\{\pm1\}}{\{q\}}\prod_{j=2}^n\frac{\{q^{j-1}\}}{\{q^{j}\}}=\frac{\pm1-(\pm1)^{-1}}{\{q\}}\prod_{j=2}^n\frac{\{q^{j-1}\}}{\{q^{j}\}}=0,\text{ for all } i>0.$$

We begin by considering $S$. As $\{\pm q^{0}\}=\{\pm 1\}=\pm1-(\pm1)^{-1}=0$, $D_{-1}=0.$ Additionally, $[2k+1]=[1]=\frac{\{q^{1}\}}{\{q\}}=1$.

At $A=\pm q$ in general, $D_{j}=[j+1]$, and $$G(n)=\prod_{j=1}^{n}\frac{[j-1]}{[j]}=\frac{[0]}{[n]}$$

Also, at $k=0$, \begin{multline*}S_{0,m}=\sigma_{0,m}(r+m)\sqrt{G(0)^2\frac{D_{2m-1}}{G(r+m+1)G(r-m)}}\frac{G(r-m)G(r+m+1)}{G(r+1)}\\=\sigma_{0,m}(r+m)\frac{\sqrt{D_{2m-1}G(r-m)G(r+m+1)}}{G(r+1)},\end{multline*}

by applying the above simplifications, without using $A=\pm q$.

$[x]$ is nonzero exactly when $x\neq 0$, so as

\begin{multline}\nonumber
\sigma_{0,m}(r+m) = (-1)^{r+m} \sqrt{[1][2m+1]} \cdot \frac{([m]!)^2[r]![r-m]!}{[r+1]![r+m+1]!} \cdot\frac{(-1)^{r+m}[r+m+1]!}{[r-m]!([m]!)^2}, \text{ for all } 0\leq k,m \leq r,
\end{multline}

and $r\geq m > 0$, $\sigma_{0,m}(r+m)$ is neither 0 and undefined for $0<m<r$, and so can be ignored within these ranges. We examine $m=0, m=r$ separately.

$G(i)$ always has exactly one factor that becomes $0$ at $A=q$ for $i\in \mathbb{Z}$ unless $i\leq 0$, and otherwise is exactly $1$, as nothing is being multiplied together. Additionally, this factor, $\{q^{0}\}$, is the same regardless of $i$. This means that at $r\neq m$, as $r,m\geq 0$, there are at least two factors in the numerator that become $0$ at $A=\pm q$, and only $1$ in the denominator, so the entire expression is zero at $m\neq r$.

At $m=r$, \begin{multline}\nonumber
\sigma_{0,m}(r+m) = \sigma_{0,r}(2r) = (-1)^{2r} \sqrt{[2r+1]} \cdot \frac{([r]!)^2[r]!}{[r+1]![2r+1]!} \cdot\frac{(-1)^{2r}[2r+1]!}{([r]!)^2} = \frac{\sqrt{[2r + 1]}}{[r+1]},\text{ for all } 0\leq k,m \leq r
\end{multline}

As the $[0]$s in the $G(i)$ cancel out at $r=m$ as $G(0)=1$, the remainder of the expression for $S_{k,m}$ simplifies to $\frac{[r+1]}{\sqrt{[2r+1]}}$, which exactly cancels out with $\sigma_{0,m}(r+m)$, giving that $S_{0,0}=1$. Therefore, $S_{0,m} = \begin{cases}1&m=r\\0&m\neq r\end{cases}$ at $A=\pm q$

The other cases (with one or both of $A=\pm \frac{1}{q}$ rather than $A=\pm q$ and $\bar{S}$ rather than $S$) follow in precisely the same way.

\end{proof}

\begin{theorem}\label{thm:main}
For all symmetric representations [r], $(A-q)(A+q)(Aq-1)(Aq+1)\mid Q^{1}(m,r),\text{ for all } a,b,m\in\mathbb{Z^{+}}.$
\end{theorem}

\begin{proof}
   Much like the proof of the proposition, the theorem is equivalent to the statements that $$\begin{array}{c}Q(2m+1,r)-Q(2m-1,r)\mid_{A=q}=0, \\ Q(2m+1,r)-Q(2m-1,r)\mid_{A=-q}=0, \\ Q(2m+1,r)-Q(2m-1,r)\mid_{A=\frac{1}{q}}=0, \\ Q(2m+1,r)-Q(2m-1,r)\mid_{A=-\frac{1}{q}}=0\end{array}$$ from the definition of $Q^{1}(m,r)$. We know that $${\cal{H}}(a,b,c,r) = Q(c,r)=\chi_{[r,0]}\sum_{x=0}^{r}\frac{1}{S_{0,x}}(\overline{S}\cdot\overline{T}^{a}\cdot S)_{0,x}(\overline{S}\cdot\overline{T}^{b}\cdot S)_{0,x}(\overline{S}\cdot\overline{T}^{c}\cdot S)_{0,x}.$$
   
   We now recall Lemma~\ref{lemma:mono}. In particular, the formula only involves the first row of $\bar{S}\cdot \bar{T}^n S$, so we prove a simple form for this. By formula~\ref{eq:T}, $\bar{T}^n$ is a diagonal matrix with $\bar{T}_{0,0}=1$. We consider each of $A=\pm q$, $A=\pm \frac{1}{q}$ together using the lemma. Therefore, $\bar{S}\cdot \bar{T}^n$ has a first row equal to $(\bar{S}\cdot \bar{T}^n)_{0,m}=\begin{cases}\bar{S}_{0,0}\bar{T}^n_{0,0}&m=0\\0&m\neq 0\end{cases}=\begin{cases}1&m=0\\0&m\neq 0\end{cases}$, from the lemma. Then, again by the lemma, $((\bar{S}\cdot \bar{T}^n)\cdot S)_{0,m}=S_{0,m}$.
   
   Therefore, \begin{equation}\label{eq:last}Q(c,r)=\chi_{[r,0]}\sum_{x=0}^{r}\frac{S_{0,x}^3}{S_{0,x}}=\chi_{[r,0]}\sum_{x=0}^{r}S_{0,x}^2=\chi_{[r,0]},\end{equation}

    which is constant. Therefore, $Q^{1}(m,r)=Q(2m+1,r)-Q(2m-1,r)=0, \text{ for all } m\in\mathbb{Z^{+}}$, regardless of $a,b$, as expected.
\end{proof}

\begin{remark}
This proof, with minimal modification, also works for any genus $g$ pretzel knot, because regardless of the power of $S_{0,x}$ in~\ref{eq:last}, so long as it is at least $1$ (which it is because it is one more than the genus), it is still a constant for constant genus, and so the first differences are always zero at $A=\pm q, \pm \frac{1}{q}$. In particular, for higher genus, we define differences such that all but the last parameter are constant.
\end{remark}

\subsection{Conjectured General Properties}

In this section, we present some conjectured properties. Each property has been verified up to $r=[5]$, $c=2$, $a=b=1$, unless otherwise specified. Some are also verified for higher $c$ for lower $r$.

\begin{conjecture}\label{conj:main}
$Q^{r}(c,r) \cdot A^{r} q^{2(r)(r-1)} = Q^{r}(c+2,r)$ for all odd integers $a,b,c$ and symmetric representations $[r]$.
\end{conjecture}
\begin{remark}
This conjecture, in addition to Lemma~\ref{lemma:permute}, would allow for computing infinitely many HOMFLY polynomials with only a finite number computed using the Racah matrices, which would result in drastic time savings. In particular, as shown in Figure ~\ref{fig:main_conj} for $r=[2]$, only $r+1$ HOMFLY polynomials would need to be computed using Racah matrices to be able to compute recursively any HOMFLY polynomial with the same $a,b$.
\end{remark}

\begin{figure}[H]
    \centering
    \begin{tikzpicture}[scale=1.1]
\node[circle, draw, scale=0.5, color=red]() at (0,0) {${\cal{H}}(1,1,1,2)$};
\node[circle, draw, scale=0.5, color=red]() at (1.75,0) {${\cal{H}}(1,1,3,2)$};
\node[circle, draw, scale=0.5, color=red]() at (3.5,0) {${\cal{H}}(1,1,5,2)$};
\draw[->] (0.5,0.5) -- (0.725,0.8);
\draw[->] (1.25,0.5) -- (1.025,0.8);
\draw[->] (2.25,0.5) -- (2.475,0.8);
\draw[->] (3,0.5) -- (2.775,0.8);
\node[circle, draw, scale=0.75, color=red]() at (0.875,1.5) {$Q^{1}(1,2)$};
\node[circle, draw, scale=0.75, color=red]() at (2.625,1.5) {$Q^{1}(3,2)$};
\draw[->] (1.375,2) -- (1.6,2.3);
\draw[->] (2.125,2) -- (1.9,2.3);
\node[circle, draw, scale=0.75, color=red]() at (1.75,3) {$Q^{2}(1,2)$};
\node[circle, draw, scale=0.75]() at (3.5,3) {$Q^{2}(3,2)$};
\draw[->] (3.125,2) -- (3.35,2.3);
\draw[->] (3.875,2) -- (3.65,2.3);
\node[circle, draw, scale=0.75]() at (4.375,1.5) {$Q^{1}(5,2)$};
\draw[->] (4,0.5) -- (4.225,0.8);
\draw[->] (4.75,0.5) -- (4.525,0.8);
\node[circle, draw, scale=0.5]() at (5.25,0) {${\cal{H}}(1,1,7,2)$};
\end{tikzpicture}
    \caption{Assuming that the red values are known, we can compute all of the black values from top to bottom with very few computations using Conjecture~\ref{conj:main}}
    \label{fig:main_conj}
\end{figure}
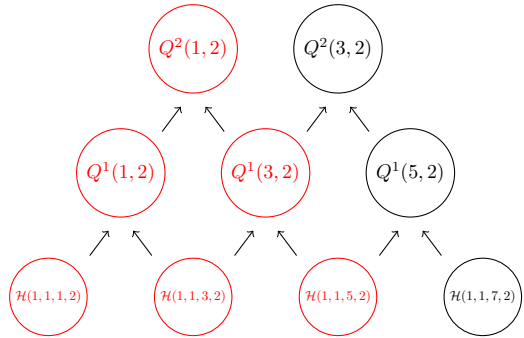

\begin{remark}
This would also imply that $Q^{r}(a,b,c,r) \mid Q^{r}(a+2,b,c,r)$ and $Q^{r}(a,b,c,r) \mid Q^{r}(a,b+2,c,r)$, for all representations $[r]$ and $a,b,c$ odd integers. This would be very powerful as it would allow one to generate the $r^{\text{th}}$ difference for every odd integer $a,b,c$, from only $r+1$ Racah matrix computations, relatively computationally inexpensively.
\end{remark}





\subsection*{Acknowledgements}

WQ is grateful to Yakov Kononov for advising him during this project.

WQ is grateful to the MIT PRIMES program for facilitating this research.

\subsection*{Data Availability Statement}

Python code and more explicitly computed values have been deposited at \\\url{https://github.com/MathTauAthogen/KnotTheory}.

\subsection*{Supporting Files}

In the arXiv submission of this paper, we include as ancillary files some explicitly computed values of differences.

\bibliographystyle{unsrt}
\nocite{*}
\bibliography{bib}

\begin{thebibliography}{10}

\bibitem{katlas}
D~Bar-Natan.
\newblock The knot atlas.
\newblock \url{http://www.katlas.org}.

\bibitem{lin2006hecke}
Xiao-Song Lin and Hao Zheng.
\newblock On the hecke algebras and the colored homfly polynomial.
\newblock {\em arXiv preprint math/0601267}, 2006.

\bibitem{defect}
Ya~Kononov and A~Morozov.
\newblock On the defect and stability of differential expansion.
\newblock {\em JETP LETT+}, 101(12):831--834, 2015.

\bibitem{col}
A~Mironov, A~Morozov, and A~Sleptsov.
\newblock Colored homfly polynomials for the pretzel knots and links.
\newblock {\em arXiv preprint arXiv:1412.8432}, 2015.

\bibitem{twist}
Andrei Mironov, Alexei Morozov, and Andrey Morozov.
\newblock On colored homfly polynomials for twist knots.
\newblock {\em MOD PHYS LETT A}, 29(34):1450183, 2014.

\bibitem{pre}
D~Galakhov, D~Melnikov, A~Mironov, A~Morozov, and A~Sleptsov.
\newblock Colored knot polynomials for pretzel knots and links of arbitrary
  genus.
\newblock {\em arXiv preprint arXiv:1412.2616}, 2014.

\bibitem{pretzel}
Symmetrically colored superpolynomials for all pretzel knots and links.
\newblock {\em To appear}.

\bibitem{all}
Oleg Alekseev and F{\'a}bio Novaes.
\newblock Wilson loop invariants from wn conformal blocks.
\newblock {\em NUCL PHYS B}, 901:461--479, 2015.

\bibitem{double}
A~Mironov, A~Morozov, An~Morozov, P~Ramadevi, and Vivek~Kumar Singh.
\newblock Colored homfly polynomials of knots presented as double fat diagrams.
\newblock {\em J HIGH ENERGY PHYS}, 2015(7):109, 2015.

\bibitem{alexandrov2014towards}
A~Alexandrov, A~Mironov, An~Morozov, et~al.
\newblock Towards matrix model representation of homfly polynomials.
\newblock {\em JETP LETT+}, 100(4):271--278, 2014.

\bibitem{tab}
A~Mironov, A~Morozov, P~Ramadevi, Vivek~Kumar Singh, and A~Sleptsov.
\newblock Tabulating knot polynomials for arborescent knots.
\newblock {\em J PHYS A-MATH THEOR}, 50(8):085201, 2017.

\bibitem{dbraid}
D~Galakhov, D~Melnikov, A~Mironov, and A~Morozov.
\newblock Knot invariants from virasoro related representation and pretzel
  knots.
\newblock {\em NUCL PHYS B}, 899:194--228, 2015.

\bibitem{dtw}
Masaya Kameyama, Satoshi Nawata, Runkai Tao, and Hao~Derrick Zhang.
\newblock Cyclotomic expansions of homfly-pt colored by rectangular young
  diagrams.
\newblock {\em arXiv preprint arXiv:1902.02275}, 2019.

\bibitem{b}
A~Morozov.
\newblock Kntz trick from arborescent calculus and the structure of
  differential expansion.
\newblock {\em arXiv preprint arXiv:2001.10254}, 2020.

\bibitem{zhao}
Yufei Zhao.
\newblock Young tableaux and the representations of the symmetric group.
\newblock {\em dimension}, 3(1):3, 2008.

\bibitem{survey}
A~Anokhina.
\newblock On r-matrix approaches to knot invariants.
\newblock {\em arXiv preprint arXiv:1412.8444}, 2014.

\bibitem{mironov2015colored}
A~Mironov, A~Morozov, An~Morozov, and A~Sleptsov.
\newblock Colored knot polynomials: Homfly in representation [2, 1].
\newblock {\em INT J MOD PHYS A}, 30(26):1550169, 2015.

\end{thebibliography}

\end{document}